\documentclass[12pt]{amsart}
\numberwithin{equation}{section}
\newtheorem{theorem}{Theorem}[section]
\newtheorem{lemma}[theorem]{Lemma}
\newtheorem{corollary}[theorem]{Corollary}
\newtheorem{proposition}[theorem]{Proposition}
\newtheorem{conjecture}[theorem]{Conjecture}

\theoremstyle{definition}
\newtheorem{definition}[theorem]{Definition}

\theoremstyle{remark}
\newtheorem{remark}[theorem]{Remark}

\newcommand{\N}{\mathbb{N}}

\usepackage{amssymb}
\usepackage{caption}
\usepackage{fullpage}
\usepackage{amsmath}
\usepackage{amsfonts}
\usepackage{latexsym}

\title{Stanley Sequences with Odd Character}
\author{Richard A. Moy}
\email{rmoy@willamette.edu}

\begin{document}
\begin{abstract}
Given a set of integers containing no $3$-term arithmetic progressions, one constructs a Stanley sequence by choosing integers greedily without forming such a progression. Independent Stanley sequences are a ``well-structured'' class of Stanley sequences with two main parameters: the character $\lambda(A)$ and the repeat factor $\rho(A)$. Rolnick conjectured that for every $\lambda\in \N_0\backslash\{1,3,5,9,11,15\}$, there exists an independent Stanley sequence $S(A)$ such that $\lambda(A)=\lambda$. This paper demonstrates that $\lambda(A)\not\in \{1,3,5,9,11,15\}$ for any independent Stanley sequence $S(A)$.
\end{abstract}
\maketitle
\vspace{-0.2in}

\section{Introduction}
Let $\N_0$ denote the set of non-negative integers. A subset of $\N_0$ is called $\ell$-free if it contains no $\ell$-term arithmetic progression. We will frequently abbreviate ``arithmetic progression'' by AP. We say a subset, or sequence of elements, of $\N_0$ is free of arithmetic progressions if it is $3$-free. In 1978, Odlyzko and Stanley \cite{OS78} used a greedy algorithm (see Definition \ref{def:greedy}) to produce arithmetic progression free sequences. Their algorithm produced sequences with two distinct growth rates -- those which are highly structured (Type I) and those which are seemingly random (Type II). These classes of Stanley sequences will be more precisely defined in Conjecture \ref{conj:growth}.

\begin{definition}\label{def:greedy}
Given a finite $3$-free set $A=\{a_0,\dots,a_n\}\subset\N_{0}$, the \textit{Stanley sequence generated by $A$} is the infinite sequence $S(A)=\{a_0,a_1,\dots\}$ defined by the following recursion. If $k\ge n$ and $a_0<\dots<a_k$ have been defined, let $a_{k+1}$ be the smallest integer $a>a_k$ such that $\{a_0,\dots, a_k\}\cup\{a\}$ is $3$-free. Though formally one writes $S(\{a_0,\dots,a_n\})$, we will frequently use the notation $S(a_0,\dots,a_n)$ instead.
\end{definition}

\begin{remark}
Without loss of generality, we may assume that every Stanley sequence begins with $0$.
\end{remark}

In Rolnick's investigation of Stanley sequences \cite{R17}, he made the following conjecture about the growth rate of the two types of Stanley sequences.

\begin{conjecture}\label{conj:growth}
Let $S(A)=(a_n)$ be a Stanley sequence. Then, for all $n$ large enough, one of the following two patterns of growth is satisfied:
\begin{itemize}
\item{
Type I: $\alpha\slash 2\le \liminf{a_n\slash n^{\log_2(3)}}\le\limsup{a_n\slash n^{\log_2(3)}}\le \alpha$, or
}
\item{
Type II: $a_n=\Theta(n^2\slash \ln(n))$.
}
\end{itemize}
\end{conjecture}

Though Type II Stanley sequences are mysterious, a great deal of progress has been made in classifying Type I Stanley sequences \cite{MR16}. In \cite{R17}, Rolnick introduced the concept of the \textit{independent Stanley sequence}. These Stanley sequences follow Type I growth and are defined as follows:
\begin{definition}
A Stanley sequence $S(A)=(a_n)$ is \textit{independent} if there exists constants $\lambda=\lambda(A)$ and $\kappa=\kappa(A)$ such that for all $k\ge \kappa$ and $0\le i<2^k$, we have 
\begin{itemize}
\item{
$a_{2^{k}+i}=a_{2^k}+a_i$
}
\item{
$a_{2^k}=2a_{2^k-1}-\lambda+1.$
}
\end{itemize}
\end{definition}
The constant $\lambda$ is called the \textit{character}, and it is easy to show that $\lambda\ge 0$ for all independent Stanley sequences. If $\kappa$ is taken as small as possible, then $a_{2^\kappa}$ is called the \textit{repeat factor}. Informally, $\kappa$ is the point at which the sequence begins its repetitive behavior. Rolnick and Venkataramana proved that every sufficiently large integer $\rho$ is the repeat factor of some independent Stanley sequence \cite{RV15}.

Rolnick also made a table \cite{R17} of independent Stanley sequences with various characters $\lambda\ge 0$. He found Stanley sequences with every character up to $75$ with the exception of those in the set $\{1,3,5,9,11,15\}$. He proved that, for an independent Stanley sequence $S(A)$, $\lambda(A)\ne 1,3$ \cite[Proposition 2.12]{R17}. In light of his observations, he made the following conjecture:

\begin{conjecture}[Conjecture 2.15, \cite{R17}]
The range of the character function is exactly the set of non-negative integers $\lambda$ that are not in the set $\{1,3,5,9,11,15\}$.
\end{conjecture}

In recent work, Sawhney \cite{S17} has shown that a positive density of even integers appear as characters of independent Stanley sequences. Analyzing the character of an independent Stanley sequence is closely related to another feature of a Stanley sequence which we introduce now.

\begin{definition}
Given a Stanley sequence $S(A)$, we define the omitted set $O(A)$ to be the set of nonnegative integers that are neither in $S(A)$ nor are covered by $S(A)$. For $O(A)\ne \emptyset$, we let $\omega(A)$ denote the largest element of $O(A)$. 
\end{definition}
\begin{remark}
The only Stanley sequence $S(A)$ where $O(A)=\emptyset$ is $S(0)$.
\end{remark}
Using this definition, one can show the following lemma.

\begin{lemma}[Lemma 2.13, \cite{R17}]\label{lemma:omitted}
If $S(A)$ is independent, then $\omega(A)<\lambda(A)$.
\end{lemma}

Since $\max(A)>\omega(A)$, the following corollary easily follows.

\begin{corollary}[Corollary 2.14, \cite{R17}]\label{conj:char}
At most finitely many independent Stanley sequences exist with a given character $\lambda$.
\end{corollary}

Using this corollary, one can show that there are no independent Stanley sequences of a given character $\lambda$ by classifying every Stanley sequence with $\omega<\lambda$. One can utilize this technique to prove that $\lambda\ne 1,3$ because every Stanley sequence with $\omega(A)<3$ is independent with $\lambda(A)\ne 1,3$. Unfortunately, this argument does not work for $\lambda=5$ because the Stanley sequence $S(0,4)$ does not \emph{appear} to be independent and experimentally exhibits Type II growth. Though no Stanley sequence, including $S(0,4)$, has been proven to follow Type II growth, we will prove that no independent Stanley sequence has character $\lambda=1,3,5,9,11,15$ by showing sequences such as $S(0,4)$ cannot be independent \emph{and} have certain characters.

\begin{theorem}\label{thm:main}
Let $S(A)$ be an independent Stanley sequence where $A$ is a finite $3$-free subset of $\N_0$. Then $\lambda(A)\not\in \{1,3,5,9,11,15 \}$.
\end{theorem}

\section{Modular sequences}
In order to prove our main result, we will use the theory of modular sequences developed in \cite{MR16} and more recently studied in \cite{ST17}. Modular sequences are a class of Stanley sequences of Type I which contains all independent Stanley sequences as a strictly smaller subset. 
\begin{definition}
Let $A$ be a set of integers and $z$ be an integer. We say that $z$ is \emph{covered} by $A$ if there exist $x,y\in A$ such that $x<y$ and $2y-x=z$. We frequently say that $z$ is covered by $x$ and $y$.

Suppose that $N$ is a positive integer. If $x,y,z\in\{0,\dots,N-1\}$ and $x\ne y$, we say they form an arithmetic progression modulo $N$, or a \emph{mod-AP} if $2y-x\equiv z\pmod N$. 

Suppose again that $N$ is a positive integer and $A\subseteq \{0,\ldots,N-1\}$. Then, we say that $z$ is \emph{covered by $A$ modulo $N$}, or \emph{mod-covered}, if there exist $x,y\in A$ with $x<y$ such that $x,y,z$ form an arithmetic progression modulo $N$.
\end{definition}

\begin{definition}
Fix a positive integer $N\ge 1$. Suppose the set $A\subset\{0,\ldots,N-1\}$ containing $0$ is 3-free modulo $N$, and all $x\in\{0,\ldots,N-1\}\backslash A$ are covered by $A$ modulo $N$. Then $A$ is said to be a \emph{modular set modulo $N$} and $S(A)$ is said to be a \emph{modular Stanley sequence modulo $N$}.
\end{definition}

Observe that the modulus $N$ of a modular Stanley sequence is analagous to the repeat factor $\rho$ of an independent Stanley sequence. One can make this statement more precise in the following proposition:

\begin{proposition}[Proposition 2.3, \cite{MR16}]
Suppose $A$ is a finite subset of $\N_0$ and suppose $S(A)$ is an independent Stanley sequence with repeat factor $\rho$. Then $S(A)$ is a modular Stanley sequence modulo $3^\ell\cdot \rho$ for some integer $\ell\ge 0$.
\end{proposition}

\begin{remark}
One can show that the modulus of a modular Stanley sequence is well-defined up to a power of $3$.
\end{remark}

Definitions made about independent Stanley sequences generalize nicely to modular Stanley sequences.

\begin{definition}\label{def:char}
Suppose that $A$ is a modular set modulo $N$. Define $\lambda(A)=2\cdot \max(A)-N+1$ and define $\omega(A)$ to be the largest element $x\in\{0,1,\dots,N-1\}\backslash A$ such that $x$ is covered by $A$ modulo $N$ but $x$ is \emph{not} covered by $A$.
\end{definition}

The definitions of $\lambda$ and $\omega$ coincide with the corresponding definitions for an independent Stanley sequence when $S(A)$ is an independent Stanley sequence.

\begin{remark}
Throughout this paper, we will repeatedly use the fact that, for a modular set $A$ modulo $N$, every element $x\in\{0,1,\dots,N-1\}\backslash A$, such that $x>\omega(A)$, is \emph{covered} by $A$ (and not merely mod-covered by $A$).
\end{remark}

\section{Proof of Main Result}
\begin{theorem}\label{thm:modular}
If $A$ is a modular set modulo $N\in\N$, then $\lambda(A)\not\in \{1,3,5,9,11,15 \}$.
\end{theorem}

Observe that this result implies Theorem \ref{thm:main} since every independent Stanley sequence is a modular Stanley sequence.

The proof of Theorem \ref{thm:modular} has been broken up into several more manageable results including Lemma \ref{lemma:char1}, Lemma \ref{lemma:char3}, Proposition \ref{prop:char5}, Proposition \ref{prop:char9}, Proposition \ref{prop:char11}, and Proposition \ref{prop:char15}. The proofs of Lemmas \ref{lemma:char1} and \ref{lemma:char3} and Proposition \ref{prop:char5} are more detailed in order to give the reader better guidance in understanding the various proof techniques. The later lemmas and propositions omit some details for brevity.

Lemma \ref{lemma:restriction}, though simple, will prove invaluable.

\begin{lemma}\label{lemma:restriction}
Suppose that $A=\{a_0,\dots,a_n\}$ with $0=a_0<\dots<a_n$ is a modular set modulo $N$ for some $N\in\N$. If $a_k>\omega(A)$, then $A=S(a_0,\dots,a_k)\cap\{0,1,\dots,N-1\}$ and $S(A)=S(a_0,\dots,a_k)$.
\end{lemma}
\begin{proof}
If $x\in\N$ with $x\le a_k$ then $x\in A$ if and only if $x\in\{a_0,\dots,a_k\}$. Therefore $S(a_0,\dots,a_k)\cap\{0,1,\dots,a_k\}=A\cap\{ a_0,\dots,a_k\}$. Now we proceed by induction. Suppose that $S(a_0,\dots,a_k)\cap\{0,1,\dots,a_m\}=A\cap\{0,1,\dots,a_m\}$ for some $k\le m<n$. If $z\in\N$ and $a_m<z<a_{m+1}$ then $z\not\in A$ and $z>\omega(A)$. Therefore, there exist $a_i,a_j\in A$ with $a_i<a_j$ such that $a_i,a_j,z$ form an AP. Since $i,j\le m$ we see that $a_i,a_j\in S(a_0,\dots,a_k)$ and therefore $z\not\in S(a_0,\dots,a_k)$. The greedy algorithm then dictates that $a_{m+1}\in S(a_0,\dots,a_k)$ and $S(a_0,\dots,a_k)\cap\{0,1,\dots,a_{m+1}\}=A\cap\{0,1,\dots,a_{m+1}\}$. By induction we have shown that $ S(a_0,\dots,a_k)\cap\{0,1,\dots,N-1 \}=A$ and $S(a_0,\dots,a_k)=S(A)$.
\end{proof}

We begin by proving a few simple lemmas. In all of these lemma, the character being investigated is odd, thus the modulus is required to be even (see Definition \ref{def:char}). Therefore, we will only consider modular sets with modulus $2N$ for some $N\in\N$.

\subsection{Characters $\lambda=1,3$}

\begin{lemma}\label{lemma:char1}
There does not exist a modular set $A$ modulo $2N$ with $\lambda(A)=1$.
\end{lemma}
\begin{proof}
Let $A$ be a modular set with modulus $2N$ where $N\in\N$ and $\lambda(A)=1$. Using the definition of $\lambda$, one finds that $\max(A)=N$. Contradiction. Every modular set contains $0$; therefore, $A$ contains the mod $2N$ arithmetic progression $0,N,0$.
\end{proof}

\begin{remark}\label{remark:mod2N}
The proof of Lemma \ref{lemma:char1} relied on the fact that if a modular set has modulus $2N$ and $x\in A$ then $x+N \pmod{2N}\not\in A$. We will use this fact repeatedly throughout the proofs of the following statements.
\end{remark}

\begin{lemma}\label{lemma:char3}
There does not exist a modular set $A$ modulo $2N$ with $\lambda(A)=3$.
\end{lemma}
\begin{proof}
Let $A$ be a modular set with modulus $2N$ where $N\in \N$ and $\lambda(A)=3$. One deduces that $\max(A)=N+1$ from the definition of $\lambda$ and $1,N\not\in A$ by \mbox{Remark \ref{remark:mod2N}}. Since $1\not\in A$, it must be mod-covered by $A$ by the definition of a modular set. That is, there exist $x,y\in A$ with $x<y$ such that $2y-x\equiv 1\pmod{2N}$. Since $0<y<2N$, one deduces that $2y-x=1$ or $2N+1$. Since $y>1$ we also know that $2y-x\ge y+1>1$ and therefore $2y-x\ne 1$. If $y<N$, then $2y-x<2N-x<2N+1$. Therefore, if $2y-x=2N+1$, then $y\ge N$ and $y=N+1=\max(A)$ necessarily. Finally, if $y=N+1$, then $2y-(2N+1)=x=1\in A$, a contradiction.
\end{proof}

Lemmas \ref{lemma:char1} and \ref{lemma:char3} were proven by Rolnick \cite{R17} in the case of independent Stanley sequences. We have proved these statements here as a warm-up for the upcoming more involved proofs.

\begin{remark}\label{remark:bigN}
In \cite{MR16}, an operation was introduced that allows one to combine the modular sets. If $A$ and $B$ are modular sets modulo $N$ and $M$ then $A\otimes B:=A+N\cdot B$ is a modular set modulo $NM$ with $\lambda(A\otimes B)=\lambda(A)+N\cdot \lambda(B)$. Through the following proofs, we will assume that $N$ is ``large.'' Let $\{0,1\}$ be the modular set of modulus $3$ with character $0$. If $A$ is a modular set modulo $2N$ then $A\otimes \{0,1\}$ is a modular set modulo $3 \cdot 2N$ with the same character $\lambda$. Thus, if we show that there is no modular set $A$ with odd character $\lambda$ of modulus $2N$ where $N>100$ (or any fixed number), then we have shown there exist no modular sets $A$ of character $\lambda$.
\end{remark}

\subsection{Character $\lambda=5$}
\begin{proposition}\label{prop:char5}
There does not exist a modular set $A$ modulo $2N$ with $\lambda(A)=5$.
\end{proposition}

We will break the proof of Proposition \ref{prop:char5} into Lemmas \ref{lemma:char5a} and \ref{lemma:char5b}.

\begin{lemma}\label{lemma:char5a}
Let $A$ be a modular set modulo $2N$ with $\lambda(A)=5$. Then, $N+1\not \in A$. 
\end{lemma}
\begin{proof}
Suppose that $N+1\in A$. Observe that $\max(A)=N+2$ and $2$ is mod-covered by $0,N+1,2$ and $1,N\not\in A$. Since $1\not\in A$, there exist $x,y\in A$ with $x<y$ such that $x,y,1$ form a mod-AP. Since $y>0$, we deduce that $2y-x=2N+1$ which further implies that $y=N+2$ and $x=3$. Since we now have $3\in A$, we see that $A=S(0,3,5)\cap\{0,1,\dots,2N-1\}$ by Lemma \ref{lemma:restriction} and therefore $S(A)=S(0,3,5)$. A quick computation shows that $S(0,3,5)=S(B)$ where $B=\{0,3,5,8\}$, a modular set modulo $9$ with character $\lambda(B)=8$. Therefore, $\lambda(A)=8$ since $S(A)=S(B)$. Contradiction.
\end{proof}

\begin{lemma}\label{lemma:char5b}
There does not exist a modular set $A$ of modulus $2N$ and $\lambda(A)=5$ with $N+1\not\in A$. 
\end{lemma}

\begin{proof}
Let $A$ be a modular set modulo $2N$ with $\lambda(A)=5$ and $N+1\not\in A$. Observe that $\max(A)=N+2$ and $2\not\in A$. Since $2\not\in A$, there exist $x,y\in A$ that mod-cover $2$. A quick computation shows that we require $x=0$ and $y=1$. Thus $1\in A$ and we see that $3$ is mod-covered by $1,N+2,3$ and $4$ is mod-covered by $0,N+2,4$. Is $5\in A$? If not, then there exist $x,y\in A$ that cover $5$ mod $2N$ since $5>\omega(A)$. This is impossible and thus $5\in A$. Hence, $S(A)=S(0,1,5)=\{0,1,5,6,8,13,\dots\}$ by Lemma \ref{lemma:restriction}.

Since $N$ is ``big,'' we know that $2N-1,2N-2,2N-3,2N-4,\dots,N+3\not\in A$. Hence, these numbers are mod-covered by $A$ and are in fact \emph{covered} by $A$ since $\omega(A)<5$. We see that $2N-1$ is covered by $5,N+2$ and $2N-2$ is covered by $6,N+2$. However, we can only cover $2N-3$ by $1,N-1$ which implies $N-1\in A$. We see $2N-4$ is covered by $8,N+2$. We know $2N-5\not\in A$ and is therefore covered by $x,y\in A$ with $x<y$. We see that $y\ne N+2$ otherwise $9\in A$, a contradiction. We also see that $y\ne N-1$ otherwise $3\in A$, a contradiction. We could cover $2N-5$ by $1,N-2$, but this is a contradiction because then $A$ contains the mod-AP $N-2,0,N+2$. Hence, $y<N-2$. However, $2N-5=2y-x\le 2(N-3)+x$, a contradiction.

Therefore, there does not exist a modular set $A$ of modulus $2N$ with $\lambda(A)=5$ with $N+1\not\in A$.
\end{proof}

The techniques from Lemmas \ref{lemma:char5a} and \ref{lemma:char5b} will be used repeatedly in the following propositions and lemmas.

\subsection{Character $\lambda=9$}

\begin{proposition}\label{prop:char9}
There does not exist a modular set $A$ modulo $2N$ with $\lambda(A)=9$.
\end{proposition}

We break the proof of Proposition \ref{prop:char9} into Lemmas \ref{lemma:char9a}, \ref{lemma:char9b}, \ref{lemma:char9c}, and \ref{lemma:char9d}. Through case by case analysis, we will eliminate all possible sets $A$.

\begin{lemma}\label{lemma:char9a}
Let $A$ be a modular set modulo $2N$ with $\lambda(A)=9$. Then $N+3\not\in A$.
\end{lemma}
\begin{proof}
Suppose that $N+3\in A$. Since $N+4=\max(A)$, we see that $N+2\not\in A$ otherwise $A$ would contain the AP $N+2,N+3,N+4$. We also see that $3,4\not\in A$ and $6$ is mod-covered by $0,N+3$ and $8$ is mod-covered by $0,N+4$. The only way to mod-cover $3$ is with $5,N+4$ and every valid way to mod-cover $4$ requires $2$; hence, $2,3\in A$. Since $0,2\in A$, we see that $N+1\not\in A$. There is no way to mod-cover $7$, so $7\in A$. We see that $9$ is covered by $5,7$ and $10$ is covered by $0,5$. However, $11$ cannot be covered, so $11\in A$ and thus we have deduced that $S(A)=S(0,2,5,7,11)=\{0,2,5,7,11,13,16,18,28,\dots \}$. 

Now we examine how $2N-1,2N-2,\dots$ are covered by $A$. We see that $2N-1$ is covered by $7,N+3$ but the only way to cover $2N-2$ is with $0,N-1$. Hence, $N-1\in A$. Similar analysis shows that $2N-3$ is covered by $11,N+4$, the element $2N-4$ is covered by $2,N-1$, and the element $2N-5$ is covered by $11,N+3$. However, $2N-6$ cannot be covered by $x<y$ using $y=N+4, N+3$ or $N-1$. We see that $y=N-3$ and $y=N-2$ are the only possible remaining choices. However, $N-3$ cannot be in $A$, otherwise $A$ contains the AP $N-3,0,N+3$. Therefore, $y=N-2$ and $x=2$ which implies that $N-2\in A$.

Further analysis shows that $2N-7,\dots,2N-13$ are covered by $A$. However, $2N-14$ cannot be covered by $x,y\in A$ with $y\in\{N-2,N-1,N+3,N+4\}$. Therefore, $y\in\{ N-7,N-6,N-5,N-4,N-3\}$. However, $N-3,N-4\not\in A$ by Remark \ref{remark:mod2N}. Furthermore, $N-5\not\in A$ otherwise $A$ would contain the AP $N-5,N-1,N+3$. Similarly, one deduces that $N-6,N-7\not\in A$. Contradiction.
\end{proof}

\begin{lemma}\label{lemma:char9b}
Let $A$ be a modular set modulo $2N$ with $\lambda(A)=9$ and $N+3\not\in A$. Then $N+1\not\in A$.
\end{lemma}
\begin{proof}
Suppose that $N+1\in A$. We see that $2$ and $8$ are mod-covered by $A$ and that $1,4\not\in A$. The only way to mod-cover $4$ is with $0,N+2$; therefore, $N+2\in A$. Observe that $3\not\in A$ otherwise $A$ would contain the mod-AP $N+2,3,N+4$. Therefore, the only way to mod-cover $1$ is with $7,N+4$ which implies $7\in A$. The only way to mod-cover $3$ is with $5,N+4$ which implies $5\in A$. Since $5,7\in A$, we see that $6\not\in A$ yet unfortunately $6$ cannot be mod-covered by $A$. Contradiction. 
\end{proof}

\begin{lemma}\label{lemma:char9c}
Let $A$ be a modular set modulo $2N$ with $\lambda(A)=9$ and $N+1,N+3\not\in A$. Then $N+2\not\in A$.
\end{lemma}
\begin{proof}
Suppose that $N+2\in A$. We see that $4,8$ are mod-covered by $A$ and $2\not\in A$. Also observe that $3\not\in A$ since otherwise $A$ would contain the mod-AP $N+2,3,N+4$. We see that $5,6\in A$ since there is no way to mod-cover them. Therefore, $3,4,7$ are mod-covered by $A$. This leaves us with no way to mod-cover $1$, so $1\in A$. We see that $9,10,11,12$ are covered and $13$ cannot be covered by $A$. Therefore, $13\in A$ and $S(A)=S(0,1,5,6,13)$.

Observe that $2N-1$ is covered by $5,N+2$ and $2N-2$ is covered by $6,N+2$ and $2N-3$. However, neither $N+2$ nor $N+4$ may be used to cover $2N-3$. Therefore, $2N-3$ is necessarily covered by $1,N-1$ which implies $N-1\in A$. However, we again deduce that $N-1,N+2,N+4$ cannot cover $2N-4$. The only way to cover $2N-4$ requires $N-2\in A$. This is a contradiction since including $N-2$ in $A$ would introduce the mod-AP $N-2,0,N+2$.  
\end{proof}

\begin{lemma}\label{lemma:char9d}
There does not exist a modular set $A$ of modulus $2N$ and $\lambda(A)=9$ with $N+1,N+2,N+3\not\in A$. 
\end{lemma}
\begin{proof}
Let $A$ be a modular set modulo $2N$ with $\lambda(A)=9$ and $N+1,N+2,N+3\not\in A$. We see that $8$ is mod-covered by $A$ and $4\not\in A$. The element $2$ is necessarily in $A$ in order to mod-cover $4$. The element $7\in A$ is needed to mod-cover $1$. We break our proof into the cases where either (Case I) $3\in A$ or (Case II) $5\in A$.

\medskip

Case I: Since $3\in A$, we see that $5$ is mod-covered by $A$ and $9\in A$ since it cannot be mod-covered by $A$. We deduce that $S(A)=S(0,2,3,7,9)=\{0,2,3,7,9,10,19,\dots\}$. Now, $2N-1$ is covered by $9,N+4$ and $2N-2$ is covered by $10,N+4$. However, $2N-3$ cannot be covered using $N+4$ and can only be covered by $1,N-1$. This is a contradiction since $1\not\in A$. 

\medskip

Case II: Since $5\in A$, we see that $3,9,10$ are mod-covered by $A$ and $11$ cannot be mod-covered by $A$. Therefore, $11\in A$ and $S(A)=S(0,2,5,7,11).$ Since $9\not\in A$, we cannot use $N+4$ to cover $2N-1$. Hence $2N-1$ cannot be covered, a contradiction.

\medskip

Therefore, a modular set $A$ of modulus $2N$ with $\lambda(A)=9$ and $N+1,N+2,N+3\not\in A$ cannot exist. 
\end{proof}

\subsection{Character $\lambda=11$}
Throughout the remainder of the paper, we will frequently write ``covered'' or ``mod-covered'' to mean ``covered by $A$'' or ``mod-covered by $A$.''

\begin{proposition}\label{prop:char11}
There does not exist a modular set $A$ modulo $2N$ with $\lambda(A)=11$.
\end{proposition}

We break the proof of Proposition \ref{prop:char11} into Lemmas \ref{lemma:char11a}, \ref{lemma:char11b}, \ref{lemma:char11c}, \ref{lemma:char11d}, \ref{lemma:char11e}, and \ref{lemma:char11f}. When $\lambda(A)=11$, observe that $\max(A)=N+5$.

\begin{lemma}\label{lemma:char11a}
Let $A$ be a modular set modulo $2N$ with $\lambda(A)=11$ with $N+2\in A$. Then $N+4\not\in A$.
\end{lemma}
\begin{proof}
Assume $N+4\in A$. Observe that $4,8,10$ are mod-covered and $2,3,5,N+3\not\in A$. Contradiction. There is no way to mod-cover $5$. 
\end{proof}

\begin{lemma}\label{lemma:char11b}
Let $A$ be a modular set modulo $2N$ with $\lambda(A)=11$ with $N+2\not\in A$. Then $N+4\not\in A$.
\end{lemma}
\begin{proof}
Assume $N+4\in A$. Observe that $8,10$ are mod-covered and $4,5,N+3\not\in A$. We see that $3\in A$ is needed to mod-cover $5$ and thus $6,7$ are also mod-covered. Since $2$ is required to mod-cover $4$, we have $2\in A$ and $1\not\in A$. We need $9\in A$ to mod-cover $1$ and $11\in A$ since it cannot be mod-covered. Therefore, $S(A)=S(0,2,3,9,11)$, a modular Stanley sequence with character $20$. This is a contradiction with $\lambda(A)=11$.
\end{proof}

Observe that Lemmas \ref{lemma:char11a} and \ref{lemma:char11b} imply that a modular set $A$ modulo $2N$ with $\lambda(A)=11$ cannot contain the element $N+4$.

\begin{lemma}\label{lemma:char11c}
Let $A$ be a modular set modulo $2N$ with $\lambda(A)=11$ with $N+4\not\in A$. Then $N+2\not\in A$.
\end{lemma}
\begin{proof}
Assume $N+2\in A$. Observe that $4,10$ are mod-covered and $2,5\not\in A$. Every possible mod-cover of $5$ includes $1$, so $1\in A$ and therefore $2,3,9$ are also mod-covered. We then see that $N+3\in A$ is required to mod-cover $5$. Therefore, $N+1\not\in A$ and $6$ is mod-covered. We cannot mod-cover $7,8,11$, so they are elements of $A$. Therefore, $S(A)=S(0,1,7,8,11)$.

We see that $2N-1,2N-2,\dots,2N-9$ are covered. However, we must include an additional element into $A$ in order to cover $2N-10$. The possible candidates are $N-5,N-4,N-3,N-2,N-1$. However, $N-5,N-3,N-2,N-1$ are not allowed for they would introduce a mod-AP into $A$. Therefore, $2N-10$ is covered by $2,N-4$. This is a contradiction with $2\not\in A$.
\end{proof}

\begin{lemma}\label{lemma:char11d}
Let $A$ be a modular set modulo $2N$ with $\lambda(A)=11$ with $N+2,N+4\not\in A$. Then $N+3\not\in A$.
\end{lemma}
\begin{proof}
Assume $N+3\in A$. Observe that $6,10$ are mod-covered and $3,4,5,N+1\not\in A$. We require $1\in A$ to mod-cover $5$, so $1\in A$ and therefore $2,9$ are also mod-covered. This is a contradiction since there is no way to mod-cover $4$.
\end{proof}

\begin{lemma}\label{lemma:char11e}
Let $A$ be a modular set modulo $2N$ with $\lambda(A)=11$ with $N+2,N+3,N+4\not\in A$. Then $N+1\not\in A$.
\end{lemma}
\begin{proof}
Assume $N+1\in A$. Observe that $2,10$ are mod-covered and $1,3,5\not\in A$. There is no way to mod-cover $5\not\in A$. Contradiction.
\end{proof}

\begin{lemma}\label{lemma:char11f}
There does not exist a modular set $A$ of modulus $2N$ and $\lambda(A)=11$ with $N+1,N+2,N+3,N+4\not\in A$. 
\end{lemma}
\begin{proof}
Let $A$ be a modular set modulo $2N$ with $\lambda(A)=11$ and $N+1,N+2,N+3,N+4\not\in A$. Observe that $10$ is mod-covered and $5\not\in A$. We see that $5$ must be covered by (Case I) $1,3$ or (Case II) $3,4$. In both cases, $3\in A$, so $6$ and $7$ are mod-covered.

\medskip

Case I: In this case $1\in A$ which implies, $2,4,9$ are mod-covered. We see that $4\in A$ since it cannot be mod-covered, so $8$ is covered. Since $11$ also cannot be mod-covered, we have $11\in A$ and $S(A)=S(0,1,3,4,11)$.

\medskip

Case II: In this case $4\in A$ which implies $5,6,8$ are mod-covered and $2\not\in A$. We see that $1$ is required to cover $2$ and in turn $2,7,9$ are mod-covered. Since $11$ cannot be mod-covered, we have $11\in A$ and $S(A)=S(0,1,3,4,11)$. 

\medskip

In both these cases, we have $S(A)=S(0,1,3,4,11)$. Now, we examine how $A$ covers $2N-1,2N-2,\dots$. The elements $2N-1,2N-2$ are covered by $11,N+5$ and $12,N+5$. However, $2N-3$ requires $N-1\in A$. Using similar reasoning, one observes that $2N-4,2N-5,2N-6$ are covered. However, covering $2N-7$ requires $N-2$ or $N-3$. We cannot include $N-3$ in $A$ otherwise it would contain the mod-AP $N-3,1,N+5$. Therefore, $N-2\in A$ and $2N-7$ is covered by $3,N-2$. We see that $2N-8$ is covered but $2N-9$ requires $N-4\in A$. Even after including $N-4\in A$, we need $N-5$ to cover $2N-10$. This is a contradiction since the set $A$ would then include the mod-AP $N-5,0,N+5$.

Therefore, there does not exist a modular set of modulus $2N$ with $\lambda(A)=11$ and $N+1,N+2,N+3,N+4\not\in A$.
\end{proof}

\subsection{Character $\lambda=15$}
\begin{proposition}\label{prop:char15}
There does not exist a modular set $A$ modulo $2N$ with $\lambda(A)=15$.
\end{proposition}

We break the proof of Proposition \ref{prop:char15} into Lemmas \ref{lemma:char15a} through \ref{lemma:char15r}. When $\lambda(A)=15$, observe that $\max(A)=N+7$.

\begin{lemma}\label{lemma:char15a}
Let $A$ be a modular set modulo $2N$ with $\lambda(A)=15$ with $N+1\in A$. Then $N+3\not\in A$.
\end{lemma}
\begin{proof}
Suppose $N+3\in A$. Then $2,6, 14$ are mod-covered and $1,3,7,N-7,N-4,N-3,N-1,N+2,N+4,N+5\not\in A$. Necessarily $7$ is mod-covered by $5,N+6\in A$; therefore, $1,9,10,12$ are mod-covered and $N-4,N-6\not\in A$. We need $11\in A$ to mod-cover $3$ which implies $8\not\in A$. Furthermore, $13\in A$ since it cannot be mod-covered. We see that $4\in A$ in order to cover $8$. We then see that $S(A)=S(0,4,5,11,13,16).$ One computes that $2N-1,\dots,2N-11$ are covered. However, $2N-12$ and therefore must be covered by $x<y$ with $x,y\in A$. However, $y=N-2$ necessarily which implies $x=8$, a contradiction with $8\not\in A$.
\end{proof}

\begin{lemma}\label{lemma:char15b}
Let $A$ be a modular set modulo $2N$ with $\lambda(A)=15$ with $N+1\in A$ and $N+3\not\in A$. Then $N+5\not\in A$.
\end{lemma}
\begin{proof}
Assume $N+5\in A$. Then $N+4,N+6\not\in A$ and $2,10,14$ are mod-covered and $1,3,4,5,6,7\not\in A$. This is a contradiction since it is impossible to mod-cover $7$.
\end{proof}

\begin{lemma}\label{lemma:char15c}
Let $A$ be a modular set modulo $2N$ with $\lambda(A)=15$ with $N+1\in A$ and $N+3,N+5\not\in A$. Then $N+6\not\in A$.
\end{lemma}
\begin{proof}
Assume $N+6\in A$. Then $2,12,14$ are mod-covered and $1,4,6,7,N-1,N+4\not\in A$. We need $5\in A$ in order to mod-cover $7$ which implies $7,9,10$ are mod-covered. We see that $8\in A$ since it cannot be mod-covered and therefore $4,6,11$ are mod-covered. Similarly, $3$ cannot be mod-covered, so $3\in A$ and thus $13$ is mod-covered. Lastly, $N+2\in A$ necessarily to mod-cover $1$.

Observe that $S(A)=S(0,3,5,8,15)=\{0,3,5,8,15,17,18,20,\dots\}$. However, there is no way to cover $2N-2\not\in A$. Contradiction.
\end{proof}

\begin{lemma}\label{lemma:char15d}
Let $A$ be a modular set modulo $2N$ with $\lambda(A)=15$ with $N+1\in A$ and $N+3,N+5,N+6\not\in A$. 
\end{lemma}
\begin{proof}
Assume $N+2\in A$. Then $2,4,14$ are mod-covered and $1,7,N+4\not\in A$. Observe that $3,5\in A$ necessarily to cover $7$. Therefore, $1,6,7,9,10,11$ are mod-covered. We see that $8\in A$ since it cannot be mod-covered. Therefore, $13$ is mod-covered. Lastly, $12\in A$ since it cannot be covered.

Therefore, $S(A)=S(0,3,5,8,12,15)$. We know that $N-1\not\in A$ otherwise this would introduce mod-AP. However, this leaves us with no way to cover $2N-2$. Contradiction.
\end{proof}

\begin{lemma}\label{lemma:char15e}
There does not exist a modular set $A$ modulo $2N$ with $\lambda(A)=15$ and $N+1\in A$ and $N+2,N+3,N+5,N+6\not\in A$.
\end{lemma}
\begin{proof}
Suppose $A$ is such a modular set. We see that $2,14$ are mod-covered and $1,7\not\in A$. Furthermore, $5\in A$ is needed to cover $7$ and $13\in A$ is needed to mod-cover $1$. Therefore, $9,10$ are covered. In order to mod-cover $7$, we require either (Case I) $3\in A$ or (Case II) $6\in A$.

\medskip

Case I: Assume $3\in A$, then $6,7,11$ are mod-covered and $4\not\in A$. This is a contradiction since there is no way to mod-cover $4$.

\medskip

Case II: Assume $6\in A$, then $7,8,12$ are mod-covered and $3,4\not\in A$. This is a contradiction since there is no way to mod-cover $4$.

\medskip

Therefore, there does not exist a modular set $A$ modulo $2N$ with $\lambda(A)=15$ such that $N+1\in A$ and $N+2,N+3,N+5,N+6\not\in A$.
\end{proof}

Observe that Lemmas \ref{lemma:char15a}, \ref{lemma:char15b}, \ref{lemma:char15c}, \ref{lemma:char15d}, and \ref{lemma:char15e} imply that $N+1\not\in A$ for a modular set $A$ modulo $2N$ with character $\lambda(A)=15$.

\begin{lemma}\label{lemma:char15f}
Let $A$ be a modular set modulo $2N$ with $\lambda(A)=15$ with $N+5\in A$ and $N+1\not\in A$. Then $N+4\not\in A$.
\end{lemma}
\begin{proof}
Suppose $N+4\in A$. Then $8,10,14$ are mod-covered and $4,5,6,7,N+3,N+6\not\in A$. In order to mod-cover $7$, we require either (Case I) $1\in A$ or (Case II) $3\in A$.

\medskip

Case I: Assume $1\in A$, then $2,7,9,13$ are mod-covered. We see $N+2\in A$ necessarily to mod-cover $4$ and thus $3$ is mod-covered. This is a contradiction since there is no way to mod-cover $5$.

\medskip

Case II: Assume $3\in A$, then $5,6,7,11$ are mod-covered and $N-1,N+2\not \in A$. We require $2\in A$ to cover $4$ Therefore, $1\not\in A$ and $12$ is mod-covered. We see $13\in A$ since it cannot be mod-covered which implies $1$ is mod-covered. Lastly $9\in A$ since it cannot be mod-covered.

\medskip

Therefore, $S(A)=S(0,2,3,9,13,19)$. We see that $N-1$ is needed to cover $2N-2$. This is a contradiction with $N-1\not\in A$.
\end{proof}

\begin{lemma}\label{lemma:char15g}
Let $A$ be a modular set modulo $2N$ with $\lambda(A)=15$ with $N+5\in A$ and $N+1,N+4\not\in A$. Then $N+2\not\in A$.
\end{lemma}
\begin{proof}
Suppose $N+2\in A$. Then $4,10,14$ are mod-covered and $2,5,6,7,N+3,N+6\not\in A$. We need $3\in A$ to mod-cover $7$ and therefore $1,6,11$ are also mod-covered. We deduce $9\in A$ to mod-cover $5$ and then deduce $8\in A$ since it cannot be mod-covered. Hence, $2,13$ are mod-covered. Lastly, $12\in A$ since it cannot be mod-covered. Therefore, $S(A)=S(0,3,8,9,12,17)$. However, $2N-1$ cannot be covered. Contradiction.  
\end{proof}

\begin{lemma}\label{lemma:char15h}
There does not exist a modular set $A$ modulo $2N$ with $\lambda(A)=15$ with $N+5\in A$ and $N+1,N+2,N+4\not\in A$.
\end{lemma}
\begin{proof}
Suppose $A$ is such a modular set. Observe that $10,14$ are mod-covered and $5,6,7,N+3,N+6\not\in A$. In order to mod-cover $7$, we require either (Case I) $1,4\in A$ or (Case II) $3\in A$.

\medskip

Case I: Assume $1,4\in A$. Therefore, $2,6,7,8,9,13$ are mod-covered. We require $3\in A$ to mod-cover $5$. Therefore, $11$ is mod-covered and $N-1\not\in A$. We have $12\in A$ since it cannot be mod-covered. We deduce that $S(A)=S(0,1,3,4,12,15)$. This is a contradiction since one cannot cover $2N-3$.

\medskip

Case II: Assume $3\in A$. Observe that $6,7,11$ are mod-covered and $N-1\not\in A$. We break this case into the following four subcases: (Case II.1) $2,9\in A$, (Case II.2) $9\in A$ and $2\not\in A$, (Case II.3) $9\not\in A$ and $1\in A$, and (Case II.4) $1,9\not\in A$.

Case II.1: In this case $2,9\in A$. We see that $1,4,5,8,12$ are mod-covered and $13\in A$ since it cannot be mod-covered. We deduce that $S(A)=S(0,2,3,9,13,19)$. This is a contradiction since there is no way to cover $2N-1$.

Case II.2: In this case $9\in A$ and $2\not\in A$. We see that $1,5$ are mod-covered. We see that $12\in A$ in order to mod-cover $2$. We deduce $4\in A$ since it cannot be mod-covered which implies $8$ is mod-covered. Lastly, $13\in A$ since it cannot be mod-covered. Therefore, $S(A)=S(0,3,4,9,12,13,16)$ and thus $S(A)$ is an independent Stanley sequence with character $\lambda(A)=24$. This is a contradiction with $\lambda(A)=15$.

Case II.3: In this case $1\in A, 9\not\in A$ and thus $2,5,9,13$ are mod-covered. We see that $4\in A$ since it cannot be mod-covered and thus $8$ is mod-covered. Lastly, we include $12\in A$ since it cannot be mod-covered. Therefore, $S(A)=S(0,1,3,4,12,15)$. This is a contradiction since there is no way to cover $2N-3$.

Case II.4: In this case $3\in A$ and $1,9\not\in A$. We see that $6,7,11$ are mod-covered. We require $4\in A$ to cover $5$. Therefore, $5,8$ are covered and $2\not\in A$. This is a contradiction since there is no way to mod-cover $9$.

Therefore, there does not exist a modular set $A$ modulo $2N$ with $\lambda(A)=15$ such that $N+5\in A$ and $N+1,N+2,N+4\not\in A$.
\end{proof}

Observe that Lemmas \ref{lemma:char15f}, \ref{lemma:char15g}, and \ref{lemma:char15h}, along with previous results, imply that $N+5\not\in A$ for a modular set $A$ modulo $2N$ with character $\lambda(A)=15$.

\begin{lemma}\label{lemma:char15i}
Let $A$ be a modular set modulo $2N$ with $\lambda(A)=15$ with $N+3\in A$ and $N+1,N+5\not\in A$. Then $N+4\not\in A$.
\end{lemma}
\begin{proof}
Suppose $N+4\in A$. Then $6,8,14$ are mod-covered and $3,4,5,7\not\in A$. We need $1\in A$ to mod-cover $7$ and therefore $2,5,7,13$ are mod-covered. We see that $9,10$ are in $A$ since they cannot be mod-covered and thus $4,11$ are mod-covered. We need $N+6\in A$ to mod-cover $3$ which implies $12$ is also mod-covered. Thus $S(A)=S(0,1,9,10,15)$ and $S(A)$ is an independent Stanley sequence with character $\lambda(A)=24$. This is a contradiction with $\lambda(A)=15$.
\end{proof}

\begin{lemma}\label{lemma:char15j}
Let $A$ be a modular set modulo $2N$ with $\lambda(A)=15$ with $N+3\in A$ and $N+1,N+4,N+5\not\in A$. Then $N+2\not\in A$.
\end{lemma}
\begin{proof}
Suppose $N+2\in A$. Then $4,6,14$ are mod-covered and $2,3,5,7\not\in A$. This is a contradiction since there is no way to mod-cover $7$. 
\end{proof}

\begin{lemma}\label{lemma:char15k}
There does not exist a modular set $A$ modulo $2N$ with $\lambda(A)=15$ with $N+3\in A$ and $N+1,N+2,N+4,N+5\not\in A$. 
\end{lemma}
\begin{proof}
We see that $6,14$ are mod-covered and $3,5,7,N-1\not\in A$. We divide the argument into the cases where either (Case I) $11\in A$ or (Case II) $11\not\in A$.

\medskip

Case I: In this case, $3$ is mod-covered. We see that $1,4\in A$ in order to mod-cover $7$. Therefore, $2,5,6,7,8,10,13$ are mod-covered and $N+6\not\in A$. We see that $9,12\in A$ since they cannot be mod-covered. Therefore, $S(A)=S(0,1,4,11,12,16)$. This is a contradiction since we cannot cover $2N-1$.

\medskip

Case II: We need $9,N+6\in A$ to mod-cover $3$ and therefore $5,12$ are also mod-covered. We require $1,4\in A$ in order to mod-cover $7$. Therefore, $2,8,10,11,13$ are also mod-covered and $N-5,N-4\not\in A$. Hence, $S(A)=S(0,1,4,9,15)$. This is a contradiction since there is no way to cover $2N-11$.

\medskip

Therefore, there does not exist a modular set $A$ modulo $2N$ with $\lambda(A)=15$ such that $N+3\in A$ and $N+1,N+2,N+4,N+5\not\in A$.
\end{proof}

Observe that Lemmas \ref{lemma:char15i}, \ref{lemma:char15j}, and \ref{lemma:char15k} imply that $N+3\not\in A$ for a modular set $A$ modulo $2N$ with character $\lambda(A)=15$.

\begin{lemma}\label{lemma:char15l}
Let $A$ be a modular set modulo $2N$ with $\lambda(A)=15$ with $N+2\in A$ and $N+1,N+3,N+5\not\in A$. Then $N+4\not\in A$.
\end{lemma}
\begin{proof}
Suppose $N+4\in A$. Then $4,8,14,N+6$ are mod-covered and $2,3,7,N-2,N-3\not\in A$. We break our proof into the cases where either (Case I) $1\in A$ or (Case II) $1\not\in A$.

\medskip

Case I: Suppose $1\in A$. Then $2,3,7,13$ are mod-covered. We then have two further subcases: (Case I.1) $5\in A$ and (Case I.2) $5\not\in A$.

Case I.1: If $5\in A$, then $9,10$ are mod-covered. We see that $6\in A$ since it cannot be mod-covered and thus $11,12$ are mod-covered. Hence, $S(A)=S(0,1,5,6,15)$. We see that $N-1\in A$ is necessary to mod-cover $2N-3$. This is a contradiction since there is no way to cover $2N-6$.

Case I.2: If $5\not\in A$, then $9\in A$ is needed to mod-cover $5$. We see that $6\in A$ since it cannot be mod-covered which implies $11,12$ are covered. Lastly, $10\in A$ since it cannot be mod-covered. Therefore, $S(A)=S(0,1,6,9,10,15)$, an independent Stanley sequence with character $\lambda(A)=24$. This is a contradiction with $\lambda(A)=15$.

\medskip

Case II: If $1\not\in A$, then one requires $5,6\in A$ to mod-cover $7$ and therefore $2,3,7,9,10,12$ are mod-covered. We require $13\in A$ to mod-cover $1$. Lastly, $11\in A$ since it cannot be mod-covered. Therefore, $S(A)=S(0,5,6,11,13,18)$. This is a contradiction since there is no way to cover $2N-6$.
\end{proof}

\begin{lemma}\label{lemma:char15m}
Let $A$ be a modular set modulo $2N$ with $\lambda(A)=15$ with $N+2\in A$ and $N+1,N+3,N+4,N+5\not\in A$. Then $N+6\not\in A$.
\end{lemma}
\begin{proof}
Suppose $N+6\in A$, then $4,12,14$ are mod-covered and $2,6,7,N-3,N-2\not\in A$. We see that $5\in A$ in order to cover $7$ and therefore $7,9,10$ are mod-covered. We conclude that $8\in A$ since it cannot be mod-covered which implies $6,11$ are mod-covered. We need $1\in A$ to cover $2$. Hence, $2,3,13$ are mod-covered and $N-4,N-5\not\in A$. Therefore, $S(A)=S(0,1,5,8,17)$. 

We need $N-1\in A$ to cover $2N-2$. This is a contradiction since there is no way to cover $2N-11$.
\end{proof}

\begin{lemma}\label{lemma:char15n}
There does not exist a modular set $A$ modulo $2N$ with $\lambda(A)=15$ with $N+2\in A$ and $N+1,N+3,N+4,N+5,N+6\not\in A$.
\end{lemma}
\begin{proof}
Suppose $A$ is such a modular set. We see that $4,14$ are mod-covered and $2,7,N-3,N-2\not\in A$. We need $5\in A$ to mod-cover $7$ and thus $9,10\not\in A$. We now break the argument up into the cases where either (Case I) $1\in A$ or (Case II) $1\not\in A$.

\medskip

Case I: Suppose $1\in A$. Then $2,3,13$ are mod-covered and $N-5\not\in A$. We see that $6\in A$ to cover $7$ and thus $8,11,12$ are also mod-covered. Therefore, $S(A)=S(0,1,5,6,15)$. We require $N-1\in A$ in order to cover $2N-3$. This is a contradiction since there is no way to cover $2N-6$.

\medskip

Case II: Suppose $1\not\in A$. We need $12\in A$ to mod-cover $2$ and thus $6\not\in A$. We need $3\in A$ to mod-cover $7$ which implies $1, 6,11$ are mod-covered and $N-1\not\in A$. We see that $8\in A$ since it cannot be mod-covered and therefore $13$ is covered. Therefore, $S(A)=S(0,3,5,8,12,15)$. This is a contradiction since one cannot cover $2N-2$. 

\medskip

Therefore, there does not exist a modular set $A$ modulo $2N$ with $\lambda(A)=15$ such that $N+2\in A$ and $N+1,N+3,N+4,N+5,N+6\not\in A$.
\end{proof}

Observe that Lemmas \ref{lemma:char15m} and \ref{lemma:char15n}, along with previous results, imply that $N+2\not\in A$ for a modular set $A$ modulo $2N$ with character $\lambda(A)=15$.

\begin{lemma}\label{lemma:char15o}
Let $A$ be a modular set modulo $2N$ with $\lambda(A)=15$ with $N+6\in A$ and $N+1,N+2,N+3,N+5\not\in A$. Then $N+4\not\in A$.
\end{lemma}
\begin{proof}
Suppose $N+4\in A$. Then $8,12,14$ are mod-covered and $4,5,6,7\not\in A$. We need $1\in A$ to mod-cover $7$ and thus $2,7,11,13$ are mod-covered and $N-2\not\in A$. Therefore we need $3\in A$ to mod-cover $6$ which implies $5,6,9$ are mod-covered and $N-1\not\in A$. Lastly, we need $10\in A$ to mod-cover $4$. Thus, $S(A)=S(0,1,3,10,15)$. This is a contradiction since we cannot cover $2N-5$.
\end{proof}

\begin{lemma}\label{lemma:char15p}
There does not exist a modular set $A$ modulo $2N$ with $\lambda(A)=15$ with $N+6\in A$ and $N+1,N+2,N+3,N+4,N+5\not\in A$.
\end{lemma}
\begin{proof}
Suppose that such a modular set $A$ exists. Observe that $12,14$ are mod-covered and $6,7,N-7,N-6\not\in A$. We break our proof up into cases where either (Case I) $5\in A$ or (Case II) $5\not\in A$.

\medskip

Case I: Since $5\in A$ then $7,9,10$ are mod-covered. We then break this case up into the subcases where (Case I.1) $4\in A$, (Case I.2) $3\in A$, or (Case I.3) $3,4\not\in A$.

Case I.1: Since $4\in A$ then $6,8$ are mod-covered and $2,3\not\in A$. We need $11\in A$ to mod-cover $3$ which implies $1$ is mod-covered. This is a contradiction since there is no way to mod-cover $2$.

Case I.2: Since $3\in A$ then $6,11$ are mod-covered and $1,4,N-1\not\in A$. We need $13\in A$ to mod-cover $1$ which then implies that $8\not\in A$. We require $2\in A$ to mod-cover $4$ and $8$ which then implies $N-2\not\in A$. Therefore, $S(A)=S(0,2,3,5,13,15)$. This is a contradiction since there is no way to cover $2N-5$.

Case I.3: Since $3,4\not\in A$, we require $8\in A$ to mod-cover $6$. Therefore, $4,11$ are also mod-covered and $2\not\in A$. This is a contradiction because there is no way to mod-cover $3$.

\medskip

Case II: Since $5\not\in A$, we require $1,4\in A$ to mod-cover $7$. Therefore, $2,8,10,11,13$ are also mod-covered and $N-5,N-4\not\in A$. One needs $3\in A$ to mod-cover $6$ which implies that $5,9$ are also mod-covered and $N-1\not\in A$. Therefore, $S(A)=S(0,1,3,4,15)$. We need $N-2\in A$ to cover $2N-8$ and $N-3\in A$ to cover $2N-9$. This is a contradiction since there is no way to cover $2N-14$.

\medskip

Therefore, there does not exist a modular set $A$ modulo $2N$ with $\lambda(A)=15$ such that $N+6\in A$ and $N+1,N+2,N+3,N+4,N+5\not\in A$.
\end{proof}

Observe that Lemmas \ref{lemma:char15o} and \ref{lemma:char15p}, along with previous results, imply that $N+6\not\in A$ for a modular set $A$ modulo $2N$ with character $\lambda(A)=15$.

\begin{lemma}\label{lemma:char15q}
There does not exist a modular set $A$ modulo $2N$ with $\lambda(A)=15$ with $N+4\in A$ and $N+1,N+2,N+3,N+5,N+6\not\in A$.
\end{lemma}
\begin{proof}
Suppose such a set $A$ exists. Observe that $8,14$ are mod-covered and $4,7\not\in A$. We break our proof up into the cases where either (Case I) $1\in A$ or (Case II) $1\not\in A$.

\medskip

Case I: Since $1\in A$, we see that $2,7,13$ are mod-covered. We need $10\in A$ to mod-cover $4$ which implies that $5\not\in A$. We see $9\in A$ since it cannot be mod-covered which implies $5,11$ are mod-covered. We have $3\in A$ since it cannot be mod-covered which implies $6$ is mod-covered and $N-1\not\in A$. Lastly, $12\in A$ since it cannot be mod-covered. Therefore, $S(A)=S(0,1,3,9,10,12,16)$. This is a contradiction since one cannot cover $2N-3$.

\medskip

Case II: Since $1\not\in A$, we need $13\in A$ to mod-cover $1$. We need $5\in A$ to mod-cover $7$ which implies $3,9,10$ are mod-covered. Therefore we need $6\in A$ to mod-cover $7$ which implies $2,12$ are also mod-covered. This is a contradiction since one cannot mod-cover $4$.

\medskip

Therefore, there does not exist a modular set $A$ modulo $2N$ with $\lambda(A)=15$ such that $N+4\in A$ and $N+1,N+2,N+3,N+5,N+6\not\in A$.
\end{proof}

\begin{lemma}\label{lemma:char15r}
There does not exist a modular set $A$ modulo $2N$ with $\lambda(A)=15$ with $N+1,N+2,N+3,N+4,N+5,N+6\not\in A$.
\end{lemma}
\begin{proof}
Suppose such a set $A$ exists. Then $14$ is mod-covered and $7\not\in A$. We break our argument into the case where either (Case I) $5\not\in A$ or (Case II) $5\in A$.

\medskip

Case I: If $5\not\in A$ then $1,4\in A$ in order to cover $7$. Therefore, $2,7,8,10,13$ are mod-covered by $A$. We now break this case up into the subcases where (Case I.1) $3\in A$ and (Case I.2) $3\not\in A$.

Case I.1: If $3\in A$, then $5,6,11$ are mod-covered by $A$. We see that $9,12\in A$ since they cannot be mod-covered. Therefore, $S(A)=S(0,1,4,6,9,12,16)$. This is a contradiction since there is no way to cover $2N-1$.

Case I.2: If $3\not\in A$ then we need $11\in A$ to mod-cover $3$ which implies $6\not\in A$. This is a contradiction because there is no way to mod-cover $6$.
 
 \medskip
 
Case II: Suppose $5\in A$. Then $9,10$ are mod-covered by $A$. We break this case up into the subcases where (Case II.1) $3\in A$, (Case II.2) $6\in A$, $3\not\in A$, and (Case II.3) $3,6\not\in A$.

Case II.1: Suppose $3\in A$. Then $6,7,11$ are mod-covered and $1,4,N-1\not\in A$. We need $13\in A$ to mod-cover $1$. We see that $2\in A$ in order to mod-cover $4$ and $8,12$ are mod-covered as well. Therefore, $S(A)=S(0,2,3,5,13,15)$. This is a contradiction since there is no way to cover $2N-3$.

Case II.2: Suppose $6\in A$ and $3\not\in A$. Then $7,8,12$ are mod-covered and $4\not\in A$. We see that $2\in A$ in order to cover $4$ and therefore $1\not\in A$. Thus, $11\in A$ in order to mod-cover $3$ and $13\in A$ in order to mod-cover $1$. Hence, $S(A)=S(0,3,5,6,11,13,18)$. This is a contradiction since there is no way to mod-cover $2N-1$.

Case II.3: Suppose $3,6\not\in A$. Therefore, $1,4\in A$ in order to cover $7$. Thus, $2,6,7,8,13$ are mod-covered and $N-5\not\in A$. We see $11\in A$ in order to mod-cover $3$ and $12\in A$ since it cannot be mod-covered. Thus, $S(A)=S(0,1,4,5,11,12,15)$. We need $N-1\in A$ to cover $2N-3$ and $N-2\in A$ to cover $2N-4$. Therefore, $N-3\not\in A$. This is a contradiction because there is no way to cover $2N-10$.

\medskip

Therefore, there does not exist a modular set $A$ modulo $2N$ with $\lambda(A)=15$ such that $N+1,N+2,N+3,N+4,N+5,N+6\not\in A$.
\end{proof}

\section{Future Directions}
Though Theorem \ref{thm:main} shows that $\lambda(A)\not\in \{1,3,5,9,11,15\}$ for all independent Stanley sequences $S(A)$, it does not show that every character value $\lambda\in \N_0\backslash\{1,3,5,9,11,15 \}$ is achieved by an independent Stanley sequence. In order to prove Conjecture \ref{conj:char}, one still needs to show that an independent Stanley sequences with character $\lambda$ exists for every $\lambda\in\N_0\backslash\{1,3,5,9,11,15 \}$. Sawhney \cite{S17} has recently shown a large subset of even numbers are characters of independent Stanley sequences. However, the case of odd character is still completely open.

\section{Acknowledgements}
The author would like to thank Joe Gallian and David Rolnick for their helpful comments on preliminary drafts of this paper.

\bibliography{oddcharacterspaper} 

\begin{thebibliography}{1}

\bibitem{MR16}
Richard~A. Moy and David Rolnick.
\newblock Novel structures in {S}tanley sequences.
\newblock {\em Discrete Math.}, 339(2):689--698, 2016.

\bibitem{OS78}
Andrew~M. Odlyzko and Richard~P. Stanley.
\newblock Some curious sequences constructed with the greedy algorithm, 1978.
\newblock Bell Laboratories internal memorandum.

\bibitem{R17}
David Rolnick.
\newblock On the classification of {S}tanley sequences.
\newblock {\em European J. Combin.}, 59:51--70, 2017.

\bibitem{RV15}
David Rolnick and Praveen~S. Venkataramana.
\newblock On the growth of {S}tanley sequences.
\newblock {\em Discrete Math.}, 338(11):1928--1937, 2015.

\bibitem{S17}
Mehtaab Sawhney.
\newblock Character values of {Stanley} sequences.
\newblock arXiv:1706.05444.

\bibitem{ST17}
Mehtaab Sawhney and Jonathan Tidor.
\newblock Two classes of modular $p$-{Stanley} sequences.
\newblock arXiv:1506.07941v2.

\end{thebibliography}
\bibliographystyle{plain}

\end{document}